\documentclass[runningheads]{llncs}
\RequirePackage[T1]{fontenc}
\RequirePackage{algorithm}
\RequirePackage{amscd}
\RequirePackage{amsfonts}
\RequirePackage{amsmath}
\RequirePackage{amssymb}
\RequirePackage{graphicx}
\RequirePackage{epstopdf}
\RequirePackage{euscript}
\RequirePackage{fancyhdr}
\RequirePackage{latexsym}
\RequirePackage{verbatim}
\RequirePackage{mathrsfs}
\RequirePackage{ifthen}
\RequirePackage{stmaryrd}

\usepackage{graphicx}
\usepackage[lofdepth,lotdepth]{subfig}
\usepackage{hyperref}
\begin{document}
\title{Average performance analysis of the stochastic gradient method for online PCA}
%
%
\author{St\'ephane Chr\'etien \inst{1} \and
Christophe Guyeux\inst{2} \and
Zhen-Wai Olivier Ho\inst{3}}
\authorrunning{S. Chr\'etien et al.}
%
\institute{National Physical Laboratory, Teddington, UK,
\email{stephane.chretien@npl.co.uk}\\
\url{https://sites.google.com/view/stephanechretien} \and
FEMTO-ST Institute, UMR 6174 CNRS, France
\email{christophe.guyeux@univ-fcomte.fr}
\and 
LMB Universit\'e de Bourgogne Franche-Comt\'e, 16 route de Gray, 25030, Besan{\c c}on, France.
\email{zhen-wai-olivier@univ-fcomte.fr}
}
\maketitle              
\begin{abstract}
This paper studies the complexity of the stochastic gradient algorithm for PCA when the data are observed in a streaming setting. We also propose an online approach for selecting  the learning rate. Simulation experiments confirm the practical relevance of the plain stochastic gradient approach and that drastic improvements can be achieved by learning the learning rate.  
\end{abstract}
\section{Introduction}
\subsection{Background}
Principal Component Analysis (PCA) 
is a paramount tool in an amazingly wide scope of applications. PCA belongs to the small list of algorithms which are extensively used in data science, medicine, finance, machine learning, etc. and the list is almost infinite. PCA is one of the basic blocks in the Geometric Science of Information. Computing singular/eigenvectors easily provides nonlinear embedding of data living on low dimensional manifold
in a straightforward manner \cite{bandeira2015ten}. The other main geometric aspect of PCA lies in the fact that eigenvectors belong to the sphere and orthogonal families of eigenvectors belong to the Stiefel manifold, an information that we should take into account when computing these objects.

In the era of Big Data though, computing a set of singular vectors might turn to be a formidable task to achieve in practice since in many cases, one is not even able to store the data matrix itself in the RAM, not even mentioning running an algorithm on it. 
In the recent years, the need to handle massive datasets has revived a tremendous soaring of online techniques and algorithms which incorporate the data in an incremental fashion.
Online convex optimisation is now a thriving field for dozens of important contributions a year, and
a remarkable impact on the way statistical estimation and machine learning is undertaken in practice \cite{hazan2016introduction,sra2012optimization,shalev2012online}. On the other hand, however, PCA lives in yet another realm, which cannot be directly reached using the techniques recently developed for convex optimisation. PCA can be performed using optimisation over the sphere and online versions of this nonconvex optimisation problem. Online or stochastic version of PCA have been extensively studied quite recently; see in particular the review \cite{cardot2015online} for a thorough analysis of the practical performance of online methods for PCA. On the theoretical side,  \cite{shamir2015stochastic,shamir2015convergence,jin2015robust,allen2016lazysvd} propose very interesting results about the behavior of stochastic gradient type algorithms with different implementation details and under various assumptions. In particular, \cite{shamir2015convergence} provides a very elegant approach to the analysis of the stochastic projected gradient descent without any assumption on the spectral gap between the largest eigenvalue and the second largest eigenvalue. 

\subsection{Our contribution}
The goal of the present paper is to study the online version of the stochastic gradient algorithm for PCA. In the setting we are interested in, the entries of the matrix we want to PCA are observed online, i.e. one empirical correlation coefficient at a time. Our two main contributions are the following.
\begin{itemize}
\item We extend the analysis presented in \cite{shamir2015convergence} to the online setting. In particular, we obtain a precise control on the average performance of the online method which does not depend on the separation between the first and the second eigenvalue. 

\item We provide a practical method to tune the learning rate, i.e. the stepsize of the gradient algorithm, based on a recent version proposed in \cite{luo2015achieving} of the Hedge Algorithm \cite{freund1997decision}. 

\end{itemize}
\subsection{Organisation of the paper}

Our main results are presented in Section \ref{mainres} where the algorithm is described and our main theorem is given. 
The proof of our main theorem is exposed in Section \ref{proof}. Implementation and numerical experiments are given in Section \ref{Imp}. In particular, a simple method for choosing the learning rate is described in 
Section \ref{Hedge}. The technical {lemm\ae} which are used in the proof of Section \ref{proof} are gathered in Section 
\ref{app} at the end of the paper.

\section{Main results}
\label{mainres}
\subsection{The stochastic projected gradient algorithm}
Given a symmetric matrix $A \in \mathbb R^{d\times d}$, the projected gradient algorithm writes 
\begin{align}
	w_{t+1} & = (I+\eta A) w_t/\Vert (I+\eta A) w_t\Vert.
\end{align}
The stochastic gradient we will study in this paper is simply defined as 
\begin{align}
w_{t+1} & = (I+\eta A_t) w_t /\Vert (I+\eta A_t) w_t\Vert
\label{stochstep}
\end{align}
where $A_t$ is defined as 
\begin{align}
    A_t &  = d^2 \ A_{i_t,j_t}\ e_{i_t} e_{j_t}^*
\end{align}
and $(i_t,j_t)$ is drawn uniformly at random from $\{1,\ldots,n\}^2$.

\subsection{Main theorem}

\begin{theorem}
\label{main}
Let $\varepsilon>0$ and assume that $\frac{1}{p}<\langle w_0, v\rangle$ for a leading eigenvector $v$ of $A$.
Define 
\begin{align}
V_T & = w_0^* \ \prod_{i=T}^1 (I+\eta A_t)^\ast ((1-\varepsilon)I-A)\prod_{t=1}^T (I+\eta A_t) \ w_0.
\end{align}
Then for $T$ satisfying 
\begin{align}
    T>\max\left(\frac{4p^2 d^2}{\varepsilon}, \frac{\log 4p\varepsilon^{-1}}{\log\left(1+\frac{\varepsilon}{pd^2}\right)}\right),
\end{align}
and $\eta=\frac{\varepsilon}{4pd^2}$, it holds that
\begin{align}
    \mathbb{E}[V_T]\leq -\frac{\varepsilon}{4p}(1+2\eta)^T.
\end{align}
\label{mainth}
\end{theorem}
Comparing our result with \cite{shamir2015convergence}, we first note that the online setting does not satisfy the hypothesis for which the framework in \cite{shamir2015convergence}
can be used. In fact, the matrices $A_t$ in the online setting are not positive semidefinite, otherwise the matrix $A$ would be a nonnegative matrix which is not the case. Regardless of this setback, \cite{shamir2015convergence} provides an upper bound for the expectation of $V_T$. In that case, \cite{shamir2015convergence} requires $T$ to be larger than $d^4p/\varepsilon^2 *c^2$ for some large enough constant $c$. Therefore, in practice, we obtain a better lower bound.
\section{Proof of the Theorem \ref{mainth}}
\label{proof}
In this section, we prove our main result, namely Theorem \ref{main}.
Define 
    \begin{align}
        B_T & = \prod_{t=T}^1 (I+\eta A_t)^* ((1-\epsilon)I-A) \prod_{t=1}^T (I+\eta A_t)
    \end{align}
    so that $V_t = w_0^*B_Tw_0$. We have the following recurrence relation. 
\begin{lemma}\label{recurrence}
    We have that
    \begin{align}
        \mathbb E [B_T] & = \mathbb E[B_{T-1}]
        +\eta \left(A^* \mathbb E[B_{T-1}]+\mathbb E[B_{T-1}] A \right) \nonumber \\
        \label{tototo}\\
        & \hspace{1cm}+\eta^2 d^2 \mathrm{diag}\left(A^*\mathrm{diag}(\mathbb E[B_{T-1}])A\right).\nonumber 
    \end{align}
\end{lemma}
\begin{proof}
Expand the recurrence relationship and take the expectation.
\end{proof}    
Expanding the recurrence in Lemma~\ref{recurrence}, we have
\begin{align}
    \mathbb{E}[V_T] & \leq w_0^\ast(I+2\eta A)^T((1-\varepsilon)I-A)w_0 \nonumber \\
    & \quad +\eta^2 d^2\sum_{i=1}^T \sum_{i=1}^T(1+2\eta)^{T-i}\Vert \mathrm{diag}(\mathbb{E}[B_{i-1}])\Vert\Vert w_0\Vert^2_2.
\label{eq:1}
\end{align}
Using an eigendecomposition of $A$ and $\Vert w_0\Vert_2^2=1$ gives
\begin{align}\label{ineq:V}
    \mathbb{E}[V_T]&\leq \sum_{j=1}^d (1+2\eta s_j)^T(1-\varepsilon-s_j)w_{0,j}^2+\eta^2 d^2\sum_{i=1}^T (1+2\eta)^{T-i}\Vert \mathrm{diag}(\mathbb{E}[B_{i-1}])\Vert.
\end{align}
where $s_1>\cdots>s_d$ denote the eigenvalue of $A$ and $w_{0,j}=\langle w_0, v_j\rangle$ denotes the $j-th$ component of $w_0$ in the basis of the eigenvectors of $A$.
Since $s_1=1$, this inequality rewrites
\begin{align}
    \mathbb{E}[V_T]&\leq -\varepsilon(1+2\eta)^Tw_{0,1}^2 + \sum_{j=2}^d(1+2\eta s_j)^T(1-\varepsilon-s_j)w_{0,j}^2 \nonumber \\
    & + \eta^2 d^2 \sum_{i=1}^T(1+2\eta)^{T-i}\Vert \mathrm{diag}(\mathbb{E}[B_{i-1}])\Vert.
\end{align}
Lemma \ref{lk} gives
\begin{align}
    \mathbb{E}[V_T] & \leq -\varepsilon(1+2\eta)^Tw_{0,1}^2 +(1+\frac{(1+2\eta(1-\varepsilon))^T}{\eta(T+1)})
    \nonumber  \\
    & +\eta^2 d^2 \sum_{i=1}^T(1+2\eta)^{T-i}\Vert \mathrm{diag}(\mathbb{E}[B_{i-1}])\Vert.
\end{align}
Factoring out $(1+2\eta)^T$, the inequality writes
\begin{align}
     \mathbb{E}[V_T] &\leq (1+2\eta)^T\Big(-\varepsilon w_{0,1}^2 +\frac{1}{(1+2\eta)^T}+\frac{(1+2\eta(1-\varepsilon))^T}{(1+2\eta)^T\eta(T+1)} \nonumber \\
     & \quad \quad \quad  +\eta^2 d^2\sum_{i=1}^T(1+2\eta)^{-i} \Vert \mathrm{diag}(\mathbb{E}[B_{i-1}])\Vert \Big)
\end{align}
Lemma \ref{norm:diag} implies that 
\begin{align}
    \Vert \mathrm{diag}(\mathbb{E}[B_k])\Vert &\le 2\frac{\eta}{\eta d^2+1}\left(\frac{1}{1-\eta(\eta d^2+2)}-\frac{1}{1-\eta} \right)(1-\varepsilon) \nonumber \\
    & \hspace{-1cm} +2\frac{\eta}{\eta d^2+1}\Big(\eta d^2 \frac{1}{1-\eta(\eta d^2+2)}\\
    & \quad +\frac{1}{1-\eta}\Big)(2-\varepsilon)+ \left(1+ \frac{\eta^2 d^2}{1-\eta(\eta d^2+2)}\right)(1-\varepsilon)\nonumber\\
    & \hspace{-2cm} \le 2\frac{\eta}{\eta d^2+1} \left(\frac{1-\varepsilon+(2-\varepsilon)\eta d^2}{1-\eta (\eta d^2+2)} +\frac{1}{1-\eta} \right)+\left(1+\frac{\eta^2 d^2}{1-\eta(\eta d^2+2)}\right)(1-\varepsilon)\nonumber\\
    & \hspace{-2cm} \le 2\frac{\eta}{\eta d^2+1} \frac{2-\varepsilon +(2-\varepsilon)\eta d^2}{1-\eta (\eta d^2 +2)} + 1+\frac{\eta^2 d^2}{1-\eta(\eta d^2+2)}
\end{align}
for all $k$. This simplifies into
\begin{align}
    \Vert \mathrm{diag}(\mathbb{E}[B_k])\Vert &\le 1+\frac{\eta^2 d^2 +4\eta }{1-\eta(\eta d^2+2)}.
\end{align}
Thus we obtain
\begin{align}
    \mathbb{E}[V_T] & \leq (1+2\eta)^T\Big( -\varepsilon w_{0,1}^2 +\frac{1}{(1+2\eta)^T}+\frac{(1+2\eta(1-\varepsilon))^T}{(1+2\eta)^T\eta(T+1)} \nonumber \\
    & \hspace{.5cm} +\eta^2 d^2\left(1+\frac{\eta^2 d^2 +4\eta }{1-\eta(\eta d^2+2)}\right)\sum_{i=1}^T(1+2\eta)^{-i}\Big)
\end{align}
Bounding $\sum_{i=1}^T (1+2\eta)^{-i}$ by its infinite series $\sum_{i^1}^\infty (1+2\eta)^{-i}=(2\eta)^{-1}$ yields
\begin{align}
    \mathbb{E}[V_T] & \leq (1+2\eta)^T\Big( -\varepsilon w_{0,1}^2 +\frac{1}{(1+2\eta)^T}+\frac{(1+2\eta(1-\varepsilon))^T}{(1+2\eta)^T\eta(T+1)}\\
    & +\eta/2 d^2\left(1+\frac{\eta^2 d^2 +4\eta }{1-\eta(\eta d^2+2)}\right)\Big) .
\end{align}
We can show that for well chosen $\eta$ and $T$, the term under parenthesis is less that $-\varepsilon/4p$. Taking for example $\eta=\frac{\varepsilon}{4C p d^2}$ for some constant $C$ such that $\left(1+\frac{\eta^2 d^2 +4\eta }{1-\eta(\eta d^2+2)}\right)\le 2$ and $T>\max(4p^2 d^2C/\varepsilon, \log (4p\varepsilon^{-1})/\log(1+\varepsilon/(Cpd^2)))$ gives the result. For a small enough $\varepsilon$, we can take $C=1$.
\section{Implementation}

\label{Imp}
\subsection{Choosing the learning rate}
\label{Hedge}

In this section, we address the question of choosing the learning rate, i.e. the stepsize $\eta$ in iterations \eqref{stochstep}. Tuning the learning rate is essential in practice as it is well known to have a huge impact on the convergence speed of the method. Our idea to tune the learning rate is as follows: 

\begin{itemize}
    \item Choose the tolerance $\epsilon\in (0,1)$, and the algorithm's parameters $R, \ K\in \mathbb N_*$, $\rho\in (0,1)$ and $\beta >0$.
    
    \vspace{.3cm}
    
    \item \textit{Burn-in period:} 
    
    \vspace{.3cm}
    
    \begin{itemize}
        \item[-] For $\eta\in \{\rho^k\}_{k=1:K}$, run $R$ gradient iterations in parallel whose iterates are denoted by $w_t^{(k,r)}$, $t=1,\ldots,B$. 
    
    \vspace{.3cm}
    
        \item[-] Define $\pi^{(k)}_0=1/K$, $k=1,\ldots,K$. For $t=1,\ldots,B$, let 
    \begin{align}
        L_t^{(k)} & = \frac2{R(R-2)} \sum_{r<r'=2,\ldots,R} \langle w_t^{(k,r)},w_t^{(k,r')}\rangle,
    \end{align}
    and for $k=1,\ldots,K$, define
    \begin{align}
        \pi_{t+1}^{(k)} & = \pi_{t}^{(k)}\ \exp 
        \left(\beta \ L_t^{(k)} \right). 
    \end{align}
    
    \item[-] Stop when $\max_{k=1,\ldots,K} \ L_{t}^{(k)}\ge 1-10 \ \epsilon$. 
    \end{itemize}
    
    \vspace{.5cm}
    
    \item \textit{After burn-in}: 
    
    \vspace{.3cm}
    
    \begin{itemize}
        \item[-] Reset $R$ to 1 and $K$ to 1.    
        
        \vspace{.3cm}
        
        \item Normalise $\pi$.
        \vspace{.3cm}
        \item[-] At each step $t=B+1,\ldots$, choose the stepsize with probability $\pi_B$.
        
        \vspace{.3cm}
        
        \item[-] Stop when $L^{(1)}_t\ge 1-\epsilon$.
    \end{itemize}

\end{itemize}

Choosing the parameter $\beta$ is more robust than choosing the learning rate. Moreover, a reasonably effective value for $\beta$ is given by (see \cite{freund1997decision}):
\begin{align}
    \beta & = \sqrt{\frac{\log(K)}{B}}.
\end{align}
\label{FShap}
\subsection{Numerical experiment}
In this section, we present a simple numerical experiment which shows that  
\begin{itemize}
    \item The stochastic gradient method actually works in practice
    \item The adaptive selection of the learning rate/stepsize described in the previous subsection actually accelerates the method's convergence drastically.
\end{itemize}
We run a simple experiment on a random i.i.d. Gaussian matrix of size $10000\times 10000$. The convergence of $(L^{(1)}_t)_{t\in \mathbb N}$ to 1 of the plain stochastic gradient method is shown in Figure \ref{sub:renonc} below. The accelerated version's convergence for the same experiment is shown in Figure \ref{sub:popul} below.
These results show that the method of the previous Section actually provides a substantial acceleration. We carefully checked that the selected learning rate is not equal to the smallest nor the largest value on the proposed grid of values between $2^{-3}$, $2^{-2}$, \ldots $2^{17}$. The observed gain in convergence speed was by a factor of 8.75. 
Extensive numerical experiment demonstrating this behavior at larger scales will be included in an expanded version of this work.

\begin{figure}[ht!] 
\begin{center} 
\subfloat[$K=1$]{ \includegraphics[scale=0.4]{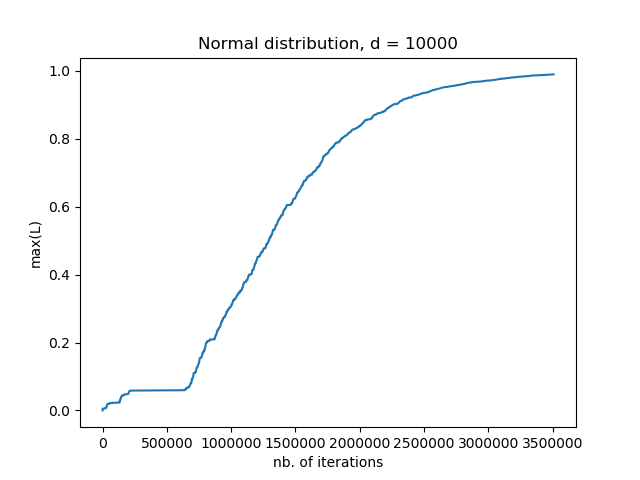} \label{sub:renonc} } 
\subfloat[$K=20$]{ \includegraphics[scale=0.4]{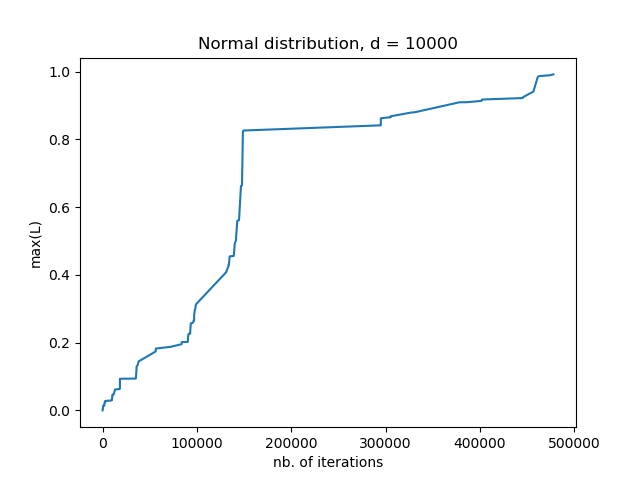} \label{sub:popul} } 
\caption{Convergence of $(L^{(1)}_t)_{t\in \mathbb N}$ as a function of the iteration index: (a) is for the case of the arbitraty choice of learning rate equal to $2^-4$ and (b) shows the behavior of the method using the learning procedure of Section \ref{FShap} for values of the learning rate equal to $2^{-3}$, $2^{-2}$, $2^{-1}$, $1$, $2$, \ldots, $2^{17}$.} 
\label{fig:my_label}
\end{center} 
\end{figure}

\appendix 

\section{Technical lemm\ae}
\label{app}
Recall that  
    \begin{align}
        B_T & = \prod_{t=T}^1 (I+\eta A_t)^* ((1-\varepsilon)I-A) \prod_{t=1}^T (I+\eta A_t).
    \end{align}
    
\begin{lemma}\label{esd}
	In the case of matrix completion, given a matrix $X$, we have
	\begin{align*}
	\mathbb{E}[A_t^\ast XA_t]=d^2 \ \mathrm{diag}(A\ \mathrm{diag}(X)A).
	\end{align*}
\end{lemma}
\begin{proof}
	The resulting matrix writes
	\begin{align*}
	A_t^\ast XA_t&=d^4 A_{ij}A_{ji}e_{j_t}e_{i_t}^\ast X e_{i_t} e_{j_t}^\ast\\
	&=d^4 A_{ij}A_{ji} X_{ii} e_{j_t} e_{j_t}^\ast.
	\end{align*}
	Therefore the expected matrix writes
	\begin{align*}
	\mathbb{E}[A_t^\ast X A_t]=d^2 \sum_{i,j}^d A_{ij}A_{ji} X_{ii} e_j e_j^\ast
	\end{align*}
	Using the symmetry of $A$ gives the result.
\end{proof}

Now our next goal is to see how $\mathrm{diag}\left(A^*\mathrm{diag}(\mathbb E[B_{T-1}])A\right)$ evolves with the iterations. For this purpose, take the diagonal of \eqref{tototo}, multiply from the left by $A^*$ and from the right by $A$ and take the diagonal of the resulting expression. 

Recall that 
\begin{align}
    \Vert A\Vert_{1\rightarrow 2} & = \max_{j=1}^d \ \Vert A_{:,j} \Vert_2. \nonumber
\end{align}

\begin{lemma} We have that 
    \begin{align}
        \Vert \mathrm{diag} \left(\mathbb E [B_{T}] \right)\Vert & \le 2\eta\ \Vert \mathbb{E}[B_{T-1}]\Vert_{1\rightarrow 2 } + (1+\eta^2d^2)\ \Vert \mathrm{diag}(\mathbb E [B_{T-1}])\Vert
        \label{dyn1}
    \end{align}
\end{lemma}
\begin{proof}
Expanding the recurrence relationship \eqref{tototo} gives
\begin{align}
    \mathrm{diag}(\mathbb{E}[B_T])&=\mathrm{diag}(\mathbb{E}[B_{T-1}])+\eta \left(\mathrm{diag}\left(A^* \mathbb E[B_{T-1}]+\mathbb E[B_{T-1}] A \right)\right) \nonumber\\
    &\hspace{1cm}+\eta^2 d^2 \mathrm{diag}\left(A^*\mathrm{diag}(\mathbb E[B_{T-1}])A\right).\nonumber 
\end{align}
For any diagonal matrix $\Delta$ and symmetric matrix $A$, we have
\begin{align}
    \Vert \mathrm{diag}( A^\ast \Delta A) \Vert \le \Vert A\Vert_{1\rightarrow 2 }^2 \Vert \Delta\Vert.
\end{align}
Therefore, by taking the operator norm on both sides of the equality, we have
\begin{align}
    \Vert\mathrm{diag}(\mathbb{E}[B_T])\Vert&\le (1+\eta^2 d^2 \Vert A\Vert_{1\rightarrow 2}^2)\Vert\mathrm{diag}(\mathbb{E}[B_{T-1}])\Vert+ 2\eta \Vert\mathrm{diag}(A^* \mathbb E[B_{T-1}])\Vert
\end{align}
We conclude using $\Vert \mathrm{diag}(A^* \mathbb \mathbb E[B_{T-1}])\Vert \le \Vert A\Vert_{1\rightarrow 2}\Vert \mathbb E[B_{T-1}]\Vert_{1\rightarrow 2}$ and $\Vert A\Vert_{1\rightarrow 2 }\le 1$.
\end{proof}
We also have to understand how the $\ell_{1\rightarrow 2}$ norm evolves. 
\begin{lemma}
    We have 
    \begin{align}
        \Vert \mathbb E [B_{T}]\Vert_{1\rightarrow 2} & \le \eta \ \Vert\mathbb E[B_{T-1}]\Vert + (1+\eta) \ \Vert \mathbb{E}[B_{T-1}]\Vert_{1\rightarrow 2 }+\eta^2 d^2\ \Vert \mathrm{diag}(\mathbb E [B_{T-1}])\Vert.
        \label{dyn2}
    \end{align}
\end{lemma}
\begin{proof} 
Expanding the recurrence relationship gives
\begin{align*}
    \Vert \mathbb E [B_T]\Vert_{1\rightarrow 2}&=\Vert \mathbb E [B_{T-1}]\Vert_{1\rightarrow2}+ \eta\ \left(\Vert A^\ast\mathbb E [B_{T-1}]\Vert_{1\rightarrow 2} +\Vert \mathbb E [B_{T-1}]^\ast A\Vert_{1\rightarrow 2}\right)\\
    &+\eta^2 d^2 \Vert \mathrm{diag}(A^\ast \mathrm{diag}(\mathbb E[B_{T-1}])A)\Vert_{1 \rightarrow 2}.
\end{align*}
For a diagonal matrix $\Delta$, we have $\Vert \Delta\Vert_{1\rightarrow 2}=\Vert \Delta\Vert$. This leads to 
\begin{align*}
    \Vert \mathbb E [B_T]\Vert_{1\rightarrow 2}&=\Vert \mathbb E [B_{T-1}]\Vert_{1\rightarrow2}+ \eta\ \left(\Vert A\Vert \Vert E [B_{T-1}]\Vert_{1\rightarrow 2} +\Vert \mathbb E [B_{T-1}]\Vert \Vert A\Vert_{1\rightarrow 2}\right)\\
    &+\eta^2d^2\Vert A\Vert_{1\rightarrow 2}^2 \Vert \mathrm{diag}(\mathbb E [B_{T-1}])\Vert.
\end{align*}
Finally, using $\Vert A\Vert_{1\rightarrow 2}\le 1$ concludes the proof.
\end{proof}
We then have to understand how the operator norm of $\mathbb E [B_T]$ evolves
\begin{lemma}
We have
\begin{align}
    \Vert \mathbb E [B_T]\Vert \le (1+2\eta) \Vert \mathbb E [B_{T-1}]\Vert + \eta^2 d^2 \ \Vert \mathrm{diag}(\mathbb E [B_{T-1}])\Vert.\label{dyn3}
\end{align}
\end{lemma}
\begin{proof}
Expanding the recurrence relationship \eqref{tototo} return
\begin{align*}
    \Vert\mathbb{E}[B_T]\Vert =\mathbb E [B_{T-1}]+\eta(\Vert A^\ast \mathbb E [B_{T-1}]\Vert +\Vert \mathbb E [B_{T-1}] A\Vert)+\eta^2 d^2 \Vert \mathrm{diag}(A^\ast\mathrm{diag}(\mathbb E [B_{T-1}])A)\Vert. 
\end{align*}
Then using similar inequalities as in the proof of the lemmas above, we have the result.
\end{proof}
\begin{lemma}\label{norm:diag} Let $\Vert A\Vert =1$, then we have
    \begin{align}
        \Vert\mathrm{diag}(A^\ast \mathrm{diag}(\mathbb{E}[B_T])A)\Vert&\le \alpha\max_j (1-\varepsilon-s_j)+\beta\Vert (1-\varepsilon)I-A\Vert_{1\rightarrow2}+\gamma\max_j (1-\varepsilon-A_{jj})
    \end{align}
    where
    \begin{align*}
        \alpha&=2\frac{\eta}{\eta d^2+1}\left(\frac{1-\eta^{T-2}(\eta d^2+2)^{T-2}}{1-\eta(\eta d^2+2)}-\frac{1-\eta^{T-2}}{1-\eta} \right)\\
        \beta&=2\frac{\eta}{\eta d^2 +1}\left(\eta d^2 \frac{1-\eta^{T-2}(\eta d^2+2)^{T-2}}{1-\eta(\eta d^2+2)}+\frac{1-\eta^{T-2}}{1-\eta} \right)\\
        \gamma&=1+\eta^2d^2 \frac{1-\eta^{T-2}(\eta d^2+2)^{T-2}}{1-\eta(\eta d^2 +2)}
    \end{align*}
\end{lemma}
\begin{proof}
Expanding the recurrence and using equations \eqref{dyn1}, \eqref{dyn2}, and \eqref{dyn3} yields the following system
\begin{align}\label{mat-ineq}
    \begin{bmatrix}
    \Vert \mathbb E[B_T]\Vert\\
    \Vert \mathbb E[B_T]\Vert_{1\rightarrow 2}\\
    \Vert \mathrm{diag}(\mathbb E[B_T])\Vert
    \end{bmatrix}\le \left(I+\eta
    \begin{bmatrix} 
    2 & 0 &\eta d^2\\
    1 &1 &\eta d^2\\
    0 & 2 & \eta d^2 
    \end{bmatrix}\right)
    \begin{bmatrix}
    \Vert \mathbb E[B_{T-1}]\Vert\\
    \Vert \mathbb E[B_{T-1}]\Vert_{1\rightarrow 2}\\
    \Vert \mathrm{diag}(\mathbb E[B_{T-1}])\Vert
    \end{bmatrix}
\end{align}
    To obtain the result, we expand the inequality by recurrence. Therefore, we are interested in computing the $T$-th power of the matrix in inequality \eqref{mat-ineq}. We have
    \begin{align}
        \left(I+\eta
    \begin{bmatrix} 
    2 & 0 &\eta d^2\\
    1 &1 &\eta d^2\\
    0 & 2 & \eta d^2 
    \end{bmatrix}\right)^T=I+\sum_{i=1}^{T} \eta^i \begin{bmatrix} 
    2 & 0 &\eta d^2\\
    1 &1 &\eta d^2\\
    0 & 2 & \eta d^2 
    \end{bmatrix}^i.
    \end{align}
    After computing the power matrices, it result that
    \begin{align}
        \Vert \mathrm{diag}(\mathbb E[B_T])\Vert & \le \sum_{i=1}^{T} \left(\eta^i \frac{2(\eta d^2 +2)^{i-1}-1}{\eta d^2 +1}\right)\Vert \mathbb{E} [B_0]\Vert \nonumber \\
        & +\sum_{i=1}^{T} \left(\eta^i \frac{2\eta d^2(\eta d^2 +2)^{i-1}+1}{\eta d^2 +1}\right)\Vert \mathbb{E}[B_0]\Vert_{1\rightarrow 2 } \nonumber\\
        & + \left(1+\eta^2 d^2\sum_{i=1}^{T} (\eta^2 d^2+2\eta)^{i-1}\right)\Vert \mathrm{diag}(\mathbb{E}[B_0])\Vert.
    \end{align}
    We conclude after computing the sums and bounding from above $\Vert \mathbb{E}[B_0]\Vert$ by $\max_{j}(1-\varepsilon -s_j)$.
\end{proof}

\begin{lemma}\label{lk}
For $\eta <1$ and $\varepsilon>0$, we have
\begin{align}
    \max_{s\in [0,1]} (1+2\eta \ s)^T(1-\varepsilon-s) \le 1+\frac{(1+2\eta(1-\varepsilon))^T}{\eta(T+1)}
\end{align}
\end{lemma}
\begin{proof}
Denote $f(s)=(1+2\eta \ s)^T(1-\varepsilon-s)$. Differentiating $f$ and setting to zero, we obtain
\begin{align*}
    &2\eta T(1+2\eta \ s)^{T-1}(1-\varepsilon-s)-(1+2\eta \ s)^T=0\\
    &\Longleftrightarrow2\eta T(1-\varepsilon-s)-(1+2\eta \ s) =0\\
    &\Longleftrightarrow\frac{T(1-\varepsilon)-1/2\eta}{T+1}=s
\end{align*}
Let $s_c=\frac{T-\varepsilon-1/2\eta}{T+1}$ denote this critical point. 
Consider the two following cases :
\begin{enumerate}
\item[-]if $s_c \notin [0,1]$, then $f$ has no critical point in the domain and therefore is maximised at either domain endpoint, i.e. $$\max_{s\in [0,1]} f(s)=\max\{ f(0)=1-\varepsilon,f(1)=-\varepsilon(1+2\eta)^T\}\le 1$$
\item[-] if $s_c \in [0,1]$, then $f$ is maximised at $s_c$ and the value of $f$ at $s_c$ is 
\begin{align*}
    &\Bigg( 1+2\eta \frac{T(1-\varepsilon)-1/2\eta}{T+1}\Bigg)^T \Bigg(1-\varepsilon-\frac{T(1-\varepsilon)-1/2\eta}{T+1}\Bigg)\\
    &=\Bigg(1+\frac{2\eta T(1-\varepsilon)-1}{T+1}\Bigg)^T \Bigg(\frac{1-\varepsilon + 1/2\eta}{T+1}\Bigg)\\
    &\le (1+2\eta(1-\varepsilon))^T \Bigg(\frac{1+1/2\eta}{T+1}\Bigg)\\
    &\le \frac{(1+2\eta (1-\varepsilon))^T}{\eta (T+1)}.
\end{align*}
\end{enumerate}
Overall, the maximum value that $f$ can reach is less than $\max\{1, \frac{(1+2\eta (1-\varepsilon))^T}{\eta (T+1)}\}\le 1+ \frac{(1+2\eta (1-\varepsilon))^T}{\eta (T+1)}\}$. Hence the result.
\end{proof}
\bibliographystyle{amsplain}
\bibliography{StochPCA}
\end{document}